\documentclass[12pt]{amsart}

\usepackage{amsfonts}
\usepackage{amsmath}
\usepackage[all]{xy}

\newcommand{\tensor}{\otimes}
\newcommand{\iso}{\cong}
\newcommand{\onto}{\to\hspace*{-.8em}\to}
\newcommand{\into}{\hookrightarrow}
\newcommand{\Frob}{\mathsf{\phi}}
\newcommand{\Adams}{\mathsf{\Psi}}

\newcommand{\Spec}{{\mbox{Spec}}}

\newcommand{\bF}{\mathbb{F}}

\newcommand{\bM}{\mathbb{M}}
\newcommand{\cO}{{\mathcal{O}}}
\newcommand{\bQ}{\mathbb{Q}}
\newcommand{\bZ}{\mathbb{Z}}

\def\lra{\longrightarrow}
\def\map#1{{\buildrel #1 \over \lra}} 

\def\length{\operatorname{length}}

\def\cL{\mathcal L}

\def\cO{\mathcal O}

\newcommand{\bP}{\mathbb{P}}
\def\m{\mathfrak m}
\def\dm{\operatorname{dim}}
\def\chr{\operatorname{char}}

\theoremstyle{plain} 
\newtheorem{thm}[equation]{Theorem}
\newtheorem{cor}[equation]{Corollary}
\newtheorem{lem}[equation]{Lemma}

\newtheorem*{state}{New Intersection Theorem}

\theoremstyle{remark}
\newtheorem{rem}[equation]{Remark}

\begin{document}

\bibliographystyle{plain}

\title[New proof of  NIT]{A new proof of the New Intersection Theorem}
\author{Greg Piepmeyer}
\address{
Department of Mathematics \\
University of Nebraska -- Lincoln \\
Lincoln \\
NE \\ 68588-0130 \\
U.S.A.}
\email{gpiepmeyer2@math.unl.edu}

\author{Mark E.~ Walker}
\thanks{The second author was partially supported by the NSA and the NSF
  (DMS-0601666).} 
\address{
Department of Mathematics \\
University of Nebraska -- Lincoln \\
Lincoln \\
NE \\ 68588-0130 \\
U.S.A.}
\email{mwalker5@math.unl.edu}

\begin{abstract}
In 1987 Roberts completed the proof of 
the New Intersection Theorem (NIT) by
settling the mixed characteristic case using local Chern characters,
as developed by Fulton and also by Roberts.  His proof has been the
only one recorded of the NIT in mixed characteristic.

This paper gives a new proof of this theorem, one which mostly
parallels Roberts' original proof, but avoids the use of local Chern
characters. Instead, the proof here uses Adams operations on
$K$-theory with supports as developed by Gillet-Soul\'e.
\end{abstract}



\maketitle

\begin{state} \label{NIT} 
Let $A$ be a (commutative, Noetherian) local
ring. If the complex of finite rank free $A$-modules
\begin{equation*}
0 \to F_n \to \cdots \to  F_1 \to F_0 \to 0
\end{equation*}
has non-zero homology of finite length, then $n \geq \dm(A)$.
\end{state}

In 1973, Peskine-Szpiro \cite{PS} proved the New Intersection
Theorem (NIT) in prime characteristic $p > 0$ using the Frobenius map.
Their work ushered characteristic $p$ methods to the forefront of
commutative algebra; by 1975, Hochster's  work \cite{H.=, H.CBMS}
established a reduction to characteristic $p > 0$ from
equicharacteristic zero to give a proof of the NIT in all
equicharacteristic rings.  In 1987, Roberts \cite{R.NIT,R.NIT2}
proved this theorem for mixed characteristic rings using 
local Chern characters.

In this paper, we give a new proof of the NIT in the mixed
characteristic case.\footnote{Our proof also applies to the
equicharacteristic $p > 0$ case, but it is considerably more
complicated than the original argument of Peskine-Spiro.}  This proof
parallels Roberts' original proof in many respects, but differs in that
it entirely avoids using local Chern theory.  Instead, we use Adams
operations on $K$-theory with supports, as developed by Gillet-Soul\'e
\cite{GS}. The difference between this proof and Roberts' proof is
much like the difference between his proof \cite{R.BAMS} of Serre's
Vanishing Conjecture and that of Gillet-Soul\'e \cite{GS}.

We wish to thank Paul Roberts for telling us about the key difficulty
in proving the NIT via Adams operations: the lack of a natural grading
for Grothendieck groups.

\section{$K$-theory, Adams operations, and the Frobenius}

Schemes in this paper are assumed to be quasi-projective over the
spectrum of a Noetherian ring.  Let $X$ be a scheme and $Z \subseteq
X$ be a closed subscheme.  Define the Grothen\-dieck group $K_0^Z(X)$
to be the abelian group generated by classes of bounded complexes of
locally free coherent
sheaves on $X$ with homology supported in $Z$, modulo the relations
coming from short exact sequences and quasi-isomorphisms. Similarly
define $G^Z_0(X)$ as the Groth\-end\-ieck group on bounded complexes
of coherent $\cO_X$-modules with homology supported in $Z$.
If $X = \Spec(A)$ and $Z = \Spec(A/I)$, then generators of $G_0^Z(X)$
(and $K_0^Z(X)$) are complexes of finitely generated (and projective)
$A$-modules.  

When $Z = X$, the support conditions on homology are
vacuous and the superscripts are omitted.  In this case, $G_0(X)$ is the
usual Groth\-end\-ieck group of coherent sheaves on $X$.  
Similarly
$K_0(X)$ is the usual Groth\-end\-ieck
group of locally free coherent sheaves on
$X$.  

If $Z'$ is a 
closed subscheme of $X$, then tensor product 
of complexes 
induces 
{\em cup product} and {\em cap product} pairings
\begin{equation*}
K_0^Z(X) \tensor K_0^{Z'}(X) \map{\cup} K_0^{Z\cap Z'}(X)
\quad
{\text{and}}
\quad
K_0^Z(X) \tensor G_0^{Z'}(X) \map{\cap} G_0^{Z \cap Z'}(X).
\end{equation*}

If $X = \Spec(A)$ for a local ring $A$ and $Z$ is the closed point, then the
Euler characteristic $\chi$ induces an isomorphism $G_0^Z(X) \iso
\bZ$.  Composing $\chi$ with cap product gives
\begin{equation*}
\chi([\bP] \cap [\bM]) =
\sideset{}{_i}\sum (-1)^i \length(H_i(\bP \otimes_A \bM)).
\end{equation*} 

For $f\colon Y \to X$ a morphism of schemes, set $W = f^{-1}(Z)$.
There is an induced pull-back map $f^*\colon K_0^Z(X) \to K_0^W(Y)$ in
$K$-theory.  When $f$ is projective (e.g., finite) there is a
push-forward map $f_*\colon G_0^W(Y) \to G_0^Z(X)$ in $G$-theory.
Push-forward along the inclusion $Z \into X$ induces an isomorphism
$\smash{G_0(Z) \map{\iso} G_0^Z(X)}$.

The projection formula relates push-forward and pull-back; we
visualize the formula as ``commutativity'' of the diagram:
\begin{equation} \label{projection}
\begin{split}
\xymatrix@R=1.41em{
K_0^Z(X) \tensor G_0(X)  \ar@<-3ex>[d]_{f^*}
                         \ar[r]^-\cap
  & G_0^Z(X) \\
K_0^W(Y) \tensor G_0(Y)  \ar@<-3ex>[u]_{f_*} 
                         \ar[r]_-\cap
  & G_0^W(Y) \ar[u]_{f_*}
}
\end{split}
\end{equation}
Thus $f_*(f^*\alpha \cap \beta) = \alpha \cap f_* \beta$ for all
$\alpha \in K_0^Z(X)$ and $\beta \in G_0(Y)$.  For $f$ a finite map
between spectra of commutative rings, the projection formula follows
from associativity and cancellation of tensor products.

The main uses of the projection formula within this paper are 
\begin{enumerate}
\item[(i)] when $Y$ is a closed subscheme of $X$ and $f$ is the
inclusion, and 

\item[(ii)] when $Y = X$ has characteristic $p > 0$ and $f$ is the
Frobenius (when it is a finite map).

\end{enumerate}

In \cite{GS}, Gillet-Soul\'e establish Adams operations $\Adams^k$ for
$k \geq 1$ on the groups $K_0^Z(X)$.  These are natural endomorphisms
extending the usual Adams operations on $K_0^X(X) = K_0(X)$. These
Adams operations have the following properties.
\begin{enumerate}
\item[(A1)] 
$\Adams^k \colon K_0^Z(X) \to K^Z_0(X)$ is an abelian group
endomorphism.

\item[(A2)] $\Adams^k(\alpha \cup \beta) = \Adams^k(\alpha) \cup
\Adams^k(\beta)$, for all $\alpha \in K_0^Z(X)$ and $\beta \in
K_0^{Z'}(X)$.

\item[(A3)]  $\Adams^k$ is functorial with respect to pull-back:  $\Adams^k
f^* \alpha = f^* \Adams^k \alpha$.  

\item[(A4)] On an affine scheme, if $K(r)$ is the Koszul complex on
one ring element $r$, then $\Adams^k([K(r)]) = k [K(r)]$.
\end{enumerate}
When $Z \hspace{-2pt} = \hspace{-2pt}X \hspace{-2pt}= \hspace{-2pt}
Z'$, then $\Adams^k$ is a ring endomorphism of \mbox{$K_0(X)
  \hspace{-2pt} = \hspace{-2pt} K_0^X(X)$}.

Our proof of the New Intersection Theorem involves passing between a
mixed characteristic ring and its reduction modulo $p$. In
both contexts the Adams operations are available while the Frobenius
map only exists in characteristic $p$. 

\begin{thm} \label{AdamsFrobenius}
Let $A$ be a Noetherian ring of characteristic $p$.  Let $X$ be
quasi-projective over ${\emph{\Spec}}(A)$ and let $Z$ be a closed
subscheme of $X$.  Write $\Frob\colon X \to X$ for the Frobenius
endomorphism.  Then the $p$-th Adams operation and the pull-back by
Frobenius coincide upon capping with classes in $G_0(X)$:
\begin{equation*}
\Adams^p \alpha \cap \beta = \Frob^* \alpha \cap \beta \in G_0^Z(X)
\end{equation*}
for all $\alpha \in K_0^Z(X)$ and all $\beta \in G_0(X)$.  
\end{thm}

\begin{rem} In the remark following \cite[4.13]{GS}, 
Gillet-Soul\'e assert (without proof) that $\Adams^p = \Frob^*$ on
  $K_0^Z(X)$. Our proof of Theorem \ref{AdamsFrobenius} does not prove
  this stronger statement.
\end{rem}

\begin{proof} 
Let $\cL_0$ and $\cL_1$ be line bundles on $X$.  By definition, a
complex of the form 
$\cdots \to 0 \to \cL_0 \to 0 \to \cdots$ or $\cdots \to 0 \to
\cL_1 \to \cL_0 \to 0 \to \cdots$, where $\cL_i$ lies in degree $i$,  
is called {\emph{elementary}}.  When
$\alpha$ is the class of an elementary complex, the theorem holds because
\begin{equation*}
\Frob^*(\alpha) = [ \cdots \to 0 \to (\cL_0)^{\tensor_p} \to 0 \to
  \cdots ] = \Adams^p(\alpha)
\end{equation*}
and 
\begin{equation*}
\Frob^*(\alpha) = [ \cdots \to 0 \to (\cL_1)^{\tensor_p}
  \map{\gamma^{\tensor_p}} (\cL_0)^{\tensor_p} \to 0 \to
  \cdots ] = \Adams^p(\alpha);
\end{equation*}
the last equality follows from \cite[3.2(c)]{KR}.  

The splitting principle for complexes \cite[18.3.12(4)]{F} gives, for
each $\alpha \in K_0^Z(X)$, a
projective morphism of schemes $\pi \colon Y \to X$ such that
\begin{enumerate}
\item[(i)] the map $\pi_* \colon G_0(X') \onto G_0(X)$ is surjective, and

\item[(ii)]  $\pi^*(\alpha)$ is a
    $\bZ$-linear combination of elementary complexes supported in
    $W := \pi^{-1}(Z)$.
\end{enumerate}
In this case, with $\beta = \pi_*\beta'$, use the projection
formula and functoriality of $\Psi^p$ and $\Frob^*$ to get
\begin{eqnarray*}
\Psi^p(\alpha) \!\cap \beta \!=\! \Psi^p(\alpha) \!\cap \pi_*\beta' \!=\!
& \hspace*{-1ex} \pi_*(\pi^*\Psi^p(\alpha) \!\cap \beta')&  
\hspace*{-1ex} \!=\! \pi_*(\Psi^p\pi^*\alpha \!\cap \beta') = \\
\pi_*(\phi^*\pi^*\alpha \!\cap \beta') \!=\! 
& \hspace*{-1ex}\pi_*(\pi^*\phi^*(\alpha) \!\cap \beta') & 
\hspace*{-1ex} \!=\! \phi^*(\alpha) \!\cap \pi_*\beta' 
\!=\! \phi^*(\alpha) \!\cap \beta,
\end{eqnarray*}
where the middle equality holds because $\pi^*(\alpha)$ decomposes
into a $\bZ$-linear combination of elementary complexes.
\end{proof}

We replace the use of local Chern characters in Roberts'
proof of the NIT with the use of the Gillet-Soul\'e Adams
operations. Local Chern characters take values in the graded Chow
group of a scheme.  By contrast, generally the Grothendieck group
$G_0(X)$ is ungraded. However, in the equicharacteristic $p > 0$ case,
the action of Frobenius provides a grading on $G_0(X)$.

\begin{lem}[see {\cite[\S 2]{Kurano}}]  \label{LemA} 
Let $(B, \m)$ be a local ring of characteristic $p > 0$ and dimension
$d$. Assume the residue field $B/\m$ is perfect and the Frobenius
map on $B$ is finite.  Set {\mbox{$Y = \emph{\Spec\,} B$}} and let
$\Frob$ be the scheme map induced by Frobenius.  Then the action of
$\Frob_*$ on the $\bQ$-vector space $G_0(Y)_\bQ := G_0(Y) \otimes_\bZ
\bQ$ is diagonalizable and its eigenvalues are a subset of $\{p^i \, |
\, i = 0, \dots, d\}$.  That is, $G_0(Y)_\bQ$ decomposes as 
\begin{equation*}
G_0(Y)_\bQ = \sideset{}{^d_0}\bigoplus V_j
\end{equation*}
where the action of $\Frob_*$ on $V_j$ is multiplication by
$p^j$. Moreover, if $B$ is a domain, then $V_d$ is one-dimensional,
spanned by $[B]$. 
\end{lem}

\begin{proof}
If $q_1, \dots, q_m$ are the minimal primes of $B$ and $Y_i = \Spec\,
B/q_i$, then the map $\bigoplus_i G_0(Y_i) \to G_0(Y)$ is onto and
commutes with $\phi_*$.   Since the quotient
of a diagonalizable endomorphism of a vector space is
diagonalizable, we may additionally assume that $B$ is a domain, say
with quotient field $E$.  We proceed by induction on $d$.  

The Localization Theorem for $G$-theory 
gives an exact sequence
\begin{equation*}
\textstyle
\bigoplus_Z G_0(Z) \to G_0(Y) \to G_0(\Spec(E)) \to 0
\end{equation*}
where $Z$ ranges over all codimension one integral closed subschemes
of $Y$; these correspond to height one prime ideals of $B$. Both maps
in this sequence commute with $\phi_*$.  

By induction, the endomorphism $\phi_*$ of $\bigoplus_Z G_0(Z)_\bQ$ is
diagonalizable and its eigenvalues are a subset of $\{p^i \, | \, i =
0, \dots, d-1\}$.  Hence the same holds for its image in
$G_0(Y)_\bQ$. By \cite[p. 125]{R.Book2}, the action of $\phi_*$ on
$G_0(E)_\bQ \iso \bQ$ is multiplication by $p^d$.  The result now
follows, since an extension of diagonalizable endomorphisms of
vector spaces is again diagonalizable provided the two sets of
eigenvalues are disjoint.
\end{proof}

So every $\beta \in G_0(Y)_\bQ$ uniquely decomposes as $\beta =
\beta_0 + \cdots + \beta_d$ such that $\phi_*(\beta_j) = p^j \beta_j$
for $j = 0, \dots, \dm(Y)$.  
Define polynomials $q_j(t) =
\sum_{0}^d a_i t^i$ by $q_j(t) = \left(\Pi \, t - p^i\right)\, /\,
\left( \Pi \, p^j - p^i \right) \in \bQ[t]$ where both products run
over the set $\{0, \ldots, d\} \setminus \{j\}$.  Then the $\beta_j$'s
are found by
\begin{equation} \label{E1}
\beta_j = q_j(\Frob_*)\beta 
=  a_0 + a_1\Frob_*\beta + \cdots + a_{d}\Frob_*^{d}\beta.
\end{equation}

\section{Dutta multiplicity vanishes on reduced complexes}

\begin{thm} \label{thm1}
Let $(A,\m)$ be a local domain with perfect residue field.  Take a
nonzero $x \in \m$ and set $B = A/x$.  Assume that $B$ is
of characteristic $p > 0$ and that the Frobenius map
is finite on $B$.  If $\bF$ is a bounded complex of finitely generated
free $A$-modules whose homology has finite length, then the Dutta
multiplicity of $\bF \otimes_A B$ is zero.
\end{thm}

\begin{cor}
The New Intersection Theorem holds if the residue characteristic is
positive.
\end{cor}

\begin{proof}[Proof of Corollary] 
Let $(A, \m, k)$ be a local ring of dimension $d$ with $\chr k > 0$.
Let $\bF$ be a complex $0 \to F_n \to \cdots \to F_1 \to F_0 \to 0$ of
finite rank free $A$-modules with nonzero homology of finite length.
Assume $n < d$.  Then for any morphism $A \to B$ with $B$ of
characteristic $p$ and $n \leq \dm B$, the Dutta multiplicity of $\bF
\otimes_A B$ is positive; see \cite{R.NIT} or \cite[7.3.5]{R.Book2}.

Our next aim is to employ the theorem.  Use a faithfully flat
extension to reduce to the case when $A$ is complete and $k$ is
algebraically closed (and hence perfect).  Kill a minimal prime $P$
with $\dm(A/P) = d$ to reduce to when $A$ is a domain; note $\bF
\tensor A/P$ remains non-exact by Nakayama's Lemma.

If $A$ has mixed characteristic, then there is a prime integer $p$
in $\m$, so take $x = p$. If $A$ is equicharacteristic, then take $x$ to
be any nonzero element in $\m$. (If $\m = 0$, there was nothing to
prove.)  Set $B = A/x$. By the Cohen Structure Theorem, $B \iso
k[[Y_1, \dots, Y_s]]/I$ and hence, in particular, the Frobenius map on
$B$ is finite.  Apply the theorem to arrive at the contradiction that
the Dutta multiplicity is also zero.
\end{proof}

\begin{proof}[Proof of Theorem] 
Let $X = \Spec(A)$, $Y = \Spec(B)$ and $\iota: Y \into X$
be the canonical closed immersion. Set $d = \dm(Y)$ so that $\dm(X) =
d+1$.  Take $\beta = [B] \in G_0(Y)_\bQ$; observe $\bF \cap \beta =
[\bF \otimes_A B] = \iota^*(\bF)$.

The complex $\bF$ has homology supported in $Z = \Spec(A/\m)$ and
$i^*(\bF)$ is supported in $i^{-1}(Z) = W$.  By definition, the
Dutta multiplicity is
\begin{equation*}
\chi_\infty(i^*\bF) 
:= 
\lim_{n \to \infty} {p^{-dn}} \cdot {\chi(\phi^n)^*(\iota^*\bF)} =
\lim_{n \to \infty} {p^{-dn}} \cdot {\chi\big((\phi^n)^*(\iota^*\bF) \cap
  \beta\big)};
\end{equation*}
the last equality holds since $- \cap \beta$ is the identity on
$G_0^W(Y)$.  The element $(\phi^n)^*(\iota^*\bF) \cap \beta$ is in
$G_0(W)$ and $\chi :G_0^W(Y) \iso \bZ$. Since the residue field is
perfect, $\phi_*$ is the identity map on $G_0^W(Y) \iso G_0(W)$, so
\begin{equation*}
\chi_\infty(\iota^*\bF) = 
\lim_{n \to \infty} {p^{-dn}} \cdot 
{\chi \big(\phi_*^n  \big(\, (\phi^n)^*(\iota^*\bF) \cap \beta \big)\, \big)}.
\end{equation*}

By Lemma \ref{LemA} there is a decomposition $\beta = \sum_0^d \beta_j$
into eigenvectors for $\Frob_*$.  The projection formula
(\ref{projection}) for $\Frob^n$ gives
\begin{eqnarray*}
{p^{-dn}}  \Frob^n_* \big( {(\Frob^*)^n \iota^*(\bF) \, \cap \,
  \beta} \big)
& = & 
{p^{-dn}}
\Frob^n_*  \big( {(\Frob^*)^n \iota^*(\bF) \, \cap \, \big(\beta_0 + 
                              \cdots + \beta_d\big)}\, \big) \\
& = & 
\iota^*(\bF) \, \cap \, p^{-nd} \Frob_*^n\big(\beta_0 + \beta_1 + \cdots
                               + \beta_d\big) \\
& = & 
\iota^*[\bF] \, \cap \, p^{-nd}\big(\beta_0 + p^n \beta_1 + \cdots 
                                 + p^{nd}\beta_d\big) \\
& \hspace*{-15em} = & \hspace*{-8em}
\big( \iota^*[\bF] \cap  p^{-nd}\beta_0 \big) + 
                 \big( \iota^*[\bF] \cap p^{n-nd}\beta_1 \big) 
               + \cdots + \big( \iota^*[\bF] \cap \beta_d
	       \big). 
\end{eqnarray*}
Applying $\chi$ and taking $n \to \infty$ gives 
\begin{equation*} \label{E3}
\chi_\infty(\iota^*\bF) = \chi\big(\iota^*(\bF) \cap \beta_d\big).
\end{equation*}

Now use (\ref{E1}) to get $\beta_d = q_d(\Frob_*)(\beta)$ where
$q_d(t) = a_0 + a_1t + \cdots + a_{d}t^{d}$ is a polynomial
with rational coefficients.  Hence
\begin{eqnarray*}
\chi_\infty(\iota^*\bF) & = &
\chi \big( \iota^*\bF \cap q_d(\Frob_*)(\beta) \big) \\
& = & 
\chi \big( \iota^*\bF \cap \sideset{}{^d_0}\sum a_j \Frob_*^j(\beta) \big) \\ 
& = & 
\sideset{}{^d_0}\sum a_j  \chi \big(\Frob_*^j \Big( (\Frob^*)^j \iota^*\bF \, \cap
\, \beta \Big)\big), \\
\end{eqnarray*}
where the last equality uses the projection formula (\ref{projection})
for the $\Frob^j$'s.  By Theorem \ref{AdamsFrobenius} and using (A3)
\begin{eqnarray*}
\chi_\infty(\iota^*\bF) & = &
\sideset{}{^d_0}\sum a_j \chi \big( \Frob_*^j 
        \Big( (\Adams^p)^j (\iota^*\bF) \, \cap \, \beta \Big) \big) \\
&  = & 
\sideset{}{^d_0}\sum a_j  \chi \big( \Frob_*^j 
        \Big( \iota^* ((\Adams^p)^j\bF) \, \cap \, \beta \Big) \big).\\
\end{eqnarray*}
For each $j$ the element $\iota^* ((\Adams^p)^j\bF) \, \cap \, \beta$
belongs to $G_0^W(Y)$, and, since the residue field is perfect,
$\Frob_*$ is the identity on $G_0^W(Y)$.  Also,
under the identifications, via $\chi$,
of $G_0^W(Y)$ and $G_0^Z(X)$ with $\bZ$, the map
 $\iota_*\colon G_0(W)
\to G_0(Z)$ is the identity map. Thus
\begin{equation*}
\chi_\infty(\iota^*\bF) = 
\sideset{}{^d_0}\sum a_j  \chi \left( \iota^* ((\Adams^p)^j\bF) \, \cap \, \beta
\right)
=
\sideset{}{^d_0}\sum a_j  \chi \big( \iota_* \Big( \iota^* ((\Adams^p)^j\bF) \,
\cap \, \beta \Big) \big)
\end{equation*}
Apply the projection formula for $\iota$ to get 
\begin{equation*}
\chi_\infty(\iota^*\bF) = 
\sideset{}{^d_0}\sum a_j  \chi  \big( ((\Adams^p)^j\bF) \,
\cap \, \iota_*(\beta) \big) 
\end{equation*}
But $\iota_*(\beta) = 0$ in $G_0(X)$ since there is the short exact
sequence 
\begin{equation*}
0 \to A \map{x} A \to B \to 0
\end{equation*}
of $A$-modules.
\end{proof}

\bibliography{NITrefs}
\end{document}